\documentclass{article}
\usepackage{amsmath,amssymb}
\usepackage{amscd}
\usepackage{comment}
\usepackage{color}
\usepackage{enumerate}
\usepackage{cases}
\usepackage{mathtools}
\usepackage[all]{xy}
\usepackage{amsthm}
\usepackage{pb-diagram}
\theoremstyle{definition}
\newtheorem{defi}{Definition}
\newtheorem{theo}[defi]{Theorem}

\newtheorem{prop}[defi]{Proposition}

\newtheorem{ex}[defi]{Example}
\newtheorem{rem}[defi]{Remark}

\def\ad{{\rm ad}}

\def\det{{\rm det}}

\def\pr{{\rm pr}}

\def\End{{\rm End}}

\def\det{{\rm det}}

\def\Tr{{\rm Tr}}
\def\diag{{\rm diag}}
\def\refe{{\rm ref}}

\def\R{{\mathbb R}}

\def\C{{\mathbb C}}

\def\M{{\cal M}}

\def\H{{\cal H}}

\def\inum{{\sqrt{-1}}}

\begin{document}

\title {Restriction of Donaldson's functional to diagonal metrics on Higgs bundles with non-holomorphic Higgs fields}
\author {Natsuo Miyatake}
\date{}
\maketitle
\begin{abstract} We consider a Higgs bundle over a compact K\"ahler manifold with a smooth, non-holomorphic Higgs field. We assume that the holomorphic vector bundle decomposes into a direct sum of holomorphic line bundles. Under an assumption on the zero set of the non-holomorphic Higgs field, we provide some necessary and sufficient conditions for Donaldson's functional which is restricted to the set of diagonal Hermitian metrics associated with a holomorphic decomposition of the vector bundle to attain a minimum. In particular, when the holomorphic vector bundle decomposes into a direct sum of holomorphic line bundles, we show that we can solve the Hermitian-Einstein equation under a strong assumption even if the Higgs field is non-holomorphic.
\end{abstract}

\section{Introduction}
A {\it Higgs bundle} over a K\"ahler manifold $(X,\omega_X)$ is a pair $(E, \Phi)$ consisting of a holomorphic vector bundle $E$ and a holomorphic section $\Phi$
of $\End E \otimes\bigwedge^{1,0}$ satisfying $\Phi\wedge\Phi=0$. The holomorphic section $\Phi$ is called a
Higgs field. In this paper, we consider Higgs bundles with non-holomorphic Higgs fields, which is not usually assumed in the definition of Higgs bundles.

The Hermitian-Einstein equation and Donaldson's functional can be defined even if the Higgs field is non-holomorphic (see Section 2). 
Therefore, the following question arises: When can the Hermitian-Einstein equation be solved, and if so, what is the meaning of the solution, and what interesting things can be done with it? Before stating the main theorem of this paper, we consider what we might be able to do with the solution if it were. Suppose that $\dim_\C X=1$. Let $(E,\Phi)$ be a Higgs bundle with a holomorphic vector bundle $E$ of $c_1(E)=0$ and a non-holomorphic Higgs field $\Phi$, and $h$ a solution to the Hermitian-Einstein equation. Then, for every open subset $U\subseteq X$ such that $\bar{\partial}\Phi=0$ on $U$, the triplets $(E\left.\right|_U,\Phi\left.\right|_U, h\left.\right|_U)$ is a harmonic bundle on $U$. For example, let $(E,\Phi)\rightarrow X$ be a usual Higgs bundle with a holomorphic Higgs field $\Phi$ and $t:X\rightarrow \C$ a smooth function. Then, in general, $t\Phi$ is not a globally holomorphic Higgs field and $(E,t\Phi, h)$ is not a harmonic bundle for a solution $h$ to the Hermitian-Einstein equation of $(E,t\Phi)$. However, on an open subset $U$ such that $t$ is constant on $U$, the triplets $(E\left.\right|_U,t\Phi\left.\right|_U, h\left.\right|_U)$ is a harmonic bundle. It seems interesting if we can do something that usual Higgs bundles cannot do by using sequences of these ``almost harmonic metrics" which are solutions of the Hermitian-Einstein equation of Higgs bundles with non-holomorphic Higgs fields. Note that for manifolds of dimension 2 or higher, deriving the existence of harmonic metrics using the vanishing Chern classes argument becomes difficult if the Higgs field is non-holomorphic.

This paper aims to provide Theorem \ref{main theorem 1} concerning Higgs bundles with non-holomorphic Higgs fields and their associated Donaldson's functionals. The proof is an application of \cite[Theorem 1]{Miy1}. We first present the exact statement of the theorem, then provide further explanation for the purpose of providing the theorem. We introduce some notations:
\begin{itemize}
\item Let $E\rightarrow (X,\omega_X)$ be a holomorphic vector bundle over a compact K\"ahler manifold and $\Phi$ a smooth section of $\End E\otimes\bigwedge^{1,0}$ satisfying $\Phi\wedge\Phi=0$. Note that we do not use the integrability condition in the following. 
We also suppose that $c_1(E)=0$ for simplicity.
\item We denote by $\Lambda_{\omega_X}$ the adjoint of $\omega_X\wedge$.
\item Suppose that the holomorphic vector bundle $E$ decomposes as $E=L_1\oplus\cdots\oplus L_r$ with holomorphic line bundles $L_1,\dots, L_r\rightarrow X$. We decompose $\Phi$ as $\Phi=\Phi_0+\sum_{i,j=1,\dots, r}\Phi_{i,j}$, where $\Phi_0$ is the diagonal part and $\Phi_{i,j}$ is a (1,0)-form which takes values in $L_j^{-1}L_i$.
\item Let $V$ be an $r-1$-dimensional real vector space defined as $V\coloneqq \{x=(x_1,\dots, x_r)\in\R^r\mid x_1+\cdots+x_r=0\}$. 
\item We define vectors $v_{i,j}\in V (i,j=1,\dots, r) $ as $v_{i,j}\coloneqq u_i-u_j$, where we denote by $u_1,\dots, u_r$ the canonical basis of $\R^r$. 
\item For each $j=1,\dots, r$, we define a real number $\gamma_j$ as $\gamma_j\coloneqq \deg_{\omega_X}(L_j)$ and we also define a vector $\gamma\in V$ as $\gamma\coloneqq (\gamma_1,\dots, \gamma_r)$.
\item We define a space $\H_E$ of Hermitian metrics as 
\begin{align*}
\H_E\coloneqq \{h\mid \text{$h$ is a smooth Hermitian metric on $E$ such that $\det(h)=1$}\}.
\end{align*}
We also define a subset $\diag_L(\H_E)\subseteq \H_E$ as
\begin{align*}
\diag_L(\H_E)\coloneqq \{h\in\H_E\mid \text{$h$ splits into $h=(h_1,\dots, h_r)$ for the decomposition}\}.
\end{align*}
\item For an $h\in\H_E$, we denote by $F_h$ the curvature of $h$, and by $\Phi^{\ast h}$ the adjoint of $\Phi$ with respect to the metric $h$.
\item We fix a metric $K=(K_1,\dots, K_r)\in\diag_L(\H_E)$. 
\end{itemize}
Then the following holds: 
\begin{theo}\label{main theorem 1} {\it Suppose that for each $i,j=1,\dots, r$, if $\Phi_{i,j}\neq 0$, then $\log|\Phi_{i,j}|_{K,\omega_X}^2$ is integrable. Then the following are equivalent:
\begin{enumerate}[(i)]
\item Donaldson's functional $\M(\cdot,K)$ restricted to $\diag_L(\H_E)$ attains a minimum;
\item For any geodesic $(h_t)_{t\in\R}$ such that $h_t\in\diag_L(\H_E)$ for all $t\in\R$, if $\M(h_t,K)$ is not a constant, then $\lim_{t\to\infty}\M(h_t,K)=\infty$;
\item There exists an $h\in\diag(\H_E)$ such that
\begin{align}
\pr(\Lambda_{\omega_X}(F_h+[\Phi\wedge\Phi^{\ast h}]))=0, \label{projection}
\end{align}
where $\pr:\End E\rightarrow \End E$ is the projection to the diagonal part;
\item Donaldson's functional $\M(\cdot, K)$ restricted to $\diag_L(\H_E)$ is bounded below 
and there exist constants $C, C^\prime, C^{\prime\prime}$ such that 
\begin{align*}
|\xi|_{L^2}\leq (\M(h,K)+C)^2+C^\prime \M(h,K)+C^{\prime\prime} \ \text{for all $h\in\diag_L(\H_E)$},
\end{align*}
where we denote by $\xi$ the pair $(f_1,\dots, f_r)$ of functions such that $h=(e^{f_1}K_1,\dots, e^{f_r}K_r)$;
\item The following holds:
\begin{align}
-\gamma\in\sum_{\substack{i,j=1,\dots, r, \\ \Phi_{i,j}\neq 0}}\R_{> 0}v_{i,j}. \label{gamma}
\end{align}
Moreover, if one of the above conditions is satisfied, a minimizer $h\in\diag_L(\H_E)$ of Donaldson's functional restricted to $\diag_L(\H_E)$ is a critical point of Donaldson's functional $\M(\cdot,K):\H_E\rightarrow \R$ if and only if the off-diagonal part of $\Lambda_{\omega_X}[\Phi\wedge\Phi^{\ast h}]$ vanishes.
\end{enumerate}
}
\end{theo}
\begin{ex} Suppose that the non-holomorphic Higgs field is of the following form:
\begin{align}
&\Phi=\left(
\begin{array}{cccc}
0 & & &\Phi_r\\
\Phi_1 & 0 & &\\
& \ddots& \ddots& \\
&&\Phi_{r-1}&0
\end{array}
\right). \label{cyclic}
\end{align}
If $\Phi_i\neq 0$ and $\log|\Phi_i|_{K,\omega_X}^2$ is integrable for all $i=1,\dots, r$, then the Hermitian-Einstein equation has a solution (see also \cite{Miy2}). 
\end{ex}

The main purpose of providing Theorem \ref{main theorem 1} is to present a specific example of Higgs bundles with non-holomorphic Higgs fields whose Hermitian-Einstein equation can be solved, along with the method to solve this equation, specifically using the variational method. We discuss the solution method, particularly the application of the variational method, in detail. The potential significance of solving the Hermitian-Einstein equation for Higgs bundles with non-holomorphic Higgs fields has been addressed earlier. Apart from that, the problem itself of finding the critical point of Donaldson's functional for a Higgs bundle with a non-holomorphic Higgs field seems interesting, particularly when considering what appropriate assumptions about the Higgs field are. In the context of Donaldson's functional, the role of the Higgs field term is to counterbalance the decrease in energy of the curvature term of the Hermitian metric. It seems natural to make some assumptions about the zeros of the Higgs field in the direction of decreasing curvature energy when attempting to locate the critical point of Donaldson's functional for Higgs bundles with non-holomorphic Higgs fields. We hope that the concrete examples we provide above will serve as a stepping stone in forming a more general statement. 

Additionally, although we depart from the main focus of this paper, which involves examining a Higgs bundle with a non-holomorphic Higgs field, it remains unclear how to demonstrate, even if the Higgs field is holomorphic,  that the functional diverges as the time parameter of the geodesic approaches infinity based on the definition of stability for a broader range of geodesics and Higgs bundles (for the case where the Higgs field is trivial, see \cite{HK1, JMS1}). Theorem \ref{main theorem 1} above was introduced to highlight the ease with which it can be shown that the functional diverges for geodesics along holomorphic decompositions (for the relationship between the stability condition and condition (\ref{gamma}), we refer the reader to \cite{Miy2}).

This paper is organized as follows: In Section 2, we present a detailed definition of the Hermitian-Einstein equation and Donaldson's functional for a Higgs bundle equipped with a non-holomorphic Higgs field, which remains identical to the conventional definition. Section 3 is dedicated to providing a proof of Theorem \ref{main theorem 1}. In Section 4, we briefly discuss the challenges associated with generalizing (ii) of Theorem \ref{main theorem 1} to cover a broader range of geodesics, especially those coupled with a non-holomorphic splitting of the vector bundle.

\begin{rem} Higgs bundles whose Higgs fields have the form given in (\ref{cyclic}) are called cyclic Higgs bundles (see \cite{Miy1, Miy2} and the references therein).\end{rem}

\begin{rem} The author is not aware of any work that extends Hitchin and Simpson's theorem \cite{Hit1, Sim1} to cases where the Higgs field is non-holomorphic. However, there is a work \cite{BT1} that studies the extension to the case where the base manifolds are not complex when the Higgs field is trivial.
\end{rem} 
\begin{rem} For the proof of the implications (i) $\Rightarrow$ (ii), (ii) $\Leftrightarrow$ (v), (i) $\Leftrightarrow$ (iii), (iv) $\Rightarrow$ (i), we do not need to assume anything about the zero set of the non-holomorphic Higgs field.
\end{rem}
\begin{rem} Hitchin section \cite{Hit1} is an example of Higgs bundles whose holomorphic vector bundle is a direct sum of holomorphic line bundles.
\end{rem}

\section{Hermitian-Einstein equation and Donaldson's functional}
In order to clarify the meaning of the terms we use and avoid confusion, this section provides precise definitions of each concept, despite them being the same as the conventional ones \cite{Don1, Kob1, Sim1}. Let $(X,\omega_X)$ be a compact K\"ahler manifold. Let $E\rightarrow X$ be a holomorphic vector bundle and $\Phi$ a smooth section of $\End E\otimes \bigwedge^{1,0}$ satisfying $\Phi\wedge\Phi=0$. 
\begin{defi}[Hermitian-Einstein equation]
We call the following PDE for a Hermitian metric $h$ on $E$ the {\it Hermitian-Einstein equation}:
\begin{align*}
\Lambda_{\omega_X}(F_h^{\perp}+[\Phi\wedge\Phi^{\ast h}])=0,
\end{align*}
where we denote by $F_h^{\perp}$ the trace-free part of the curvature.
\end{defi}
\begin{defi}[Donaldson's functional]\label{Donaldson's functional} For smooth Hermitian metrics $h$ and $K$ on $E$ such that $\det(h)=\det(K)$, we define a real number $\M(h,K)$ as follows:
\begin{align*}
\M(h,K)\coloneqq \int_0^1dt\int_X\inum \Tr(\Lambda_{\omega_X}(F_{h_t}^\perp+[\Phi\wedge\Phi^{\ast h_t}])g_t^{-1}\partial_tg_t),
\end{align*}
where $(h_t)_{0\leq t\leq 1}$ is a piecewise smooth family of Hermitian metrics such that $\det(h_t)=\det(K)$ and $g_t:E\rightarrow E$ is a unique Hermitian endmorphism with respect to $K$ and $h_t$ satisfying $K(g_t\cdot,\cdot)=h_t(\cdot, \cdot)$. As with the usual Donalson's functional (cf. \cite{Don1, Kob1, Sim1}), $\M(h,K)$ does not depend on the choice of the path $(h_t)_{0\leq t\leq 1}$. We call $\M(\cdot,\cdot)$ {\it Donalson's functional}.
\end{defi}
As with the usual Donaldson's functional, the following holds:
\begin{prop}\label{convex}{\it Let $K$ be a smooth metric on $E$ and $(h_t)_{t\in \R}$ a smooth family of Hermitian metrics such that $\det(h_t)=\det(K)$ for all $t\in\R$. We denote by $g_t:E\rightarrow E$ the unique Hermitian endmorphism with respect to $K$ and $h_t$ satisfying $K(g_t\cdot,\cdot)=h_t(\cdot, \cdot)$. Then the following holds:
\begin{enumerate}
\item The following holds:
\begin{align*}
\frac{d}{dt}\M(h_t,K)=\int_X\inum\Tr(\Lambda_{\omega_X}(F_{h_t}^\perp+[\Phi\wedge\Phi^{\ast h_t}])g_t^{-1}\partial_tg_t).
\end{align*}
\item Suppose that there exists a Hermitian endmorophism $s:E\rightarrow E$ with respect to $K$ such that $g_t=e^{ts}$ for all $t\in\R$. Then the following holds:
\begin{align*}
\frac{d^2}{dt^2}\M(h_t,K)=\int_X|(\bar{\partial}+\ad(\Phi))s|_{h_t}^2,
\end{align*}
where $\ad(\Phi)$ denotes $[\Phi, \cdot]$.
\end{enumerate}
}
\end{prop}
\section{Proof}
\begin{proof}[Proof of Theorem \ref{main theorem 1}]
Proof of Theorem \ref{main theorem 1} is obtained by applying \cite[Theorem 1]{Miy1}. Consider condition (iii) of Theorem \ref{main theorem 1}. Equation (\ref{projection}) for a metric $(e^{f_1}K_1,\dots, e^{f_r}K_r)$ is the following:
\begin{align}
\Delta_{\omega_X}\xi+\sum_{j=1}^r4|\Phi_{i,j}|^2_{K,\omega_X}e^{(v_{i,j}, \xi)}v_{i,j}=-2\inum \Lambda_{\omega_X}F_K, \label{mu}
\end{align}
where $\Delta_{\omega_X}$ denotes the geometric Laplacian, and $\xi$ is defined as $\xi\coloneqq (f_1,\dots, f_r)$. Equation (\ref{mu}) is a special case of equations considered in \cite{Miy1}, and Donalson's functional restricted to $\diag_L(\H_E)$ coincides with the functional introduced in \cite{Miy1}. Then from \cite[Theorem 1]{Miy1} and its proof, one can check that conditions in Theorem \ref{main theorem 1} are equivalent.
\end{proof}

\section{Non-holomorphic splittings} 
As noted in Section 1, it is uncertain how, even when the Higgs field is holomorphic, one can demonstrate from the definition of stability that the functional diverges as the time of the geodesic approaches infinity. This applies to a more general geodesic and a more general Higgs bundle. The same uncertainty remains even if the vector bundle decomposes into a direct sum of holomorphic line bundles, and the geodesic closely approximates one along a holomorphic splitting of the vector bundle.  In this section, we will briefly explore what might happen when the vector bundle decomposes into a direct sum of holomorphic line bundles and we attempt to demonstrate, naively, that the functional diverges for geodesics along non-holomorphic splittings of the vector bundle. We introduce some notations:
\begin{itemize}
\item Let $M_1,\dots, M_r$ be smooth sublinebundles of $E$ such that $E=M_1\oplus\cdots\oplus M_r$. We also fix a holomorphic splitting of $E$: $E=L_1\oplus\cdots\oplus L_r$. Although the following theorem holds no matter what these splittings are, we have in mind the situation where these two splittings are in some sense very close.
\item Let $\diag_L(\H_E)$ (resp. $\diag_M(\H_E)$) be the space of diagonal metrics concerning the decomposition $E=L_1\oplus\cdots \oplus L_r$ (resp. $E=M_1\oplus\cdots\oplus M_r$). 
\item We fix an initial metric $K=(K_1,\dots, K_r)\in \diag_M(\H_E)$. We also fix a metric $h_\refe\in \diag_L(\H_E)$.
\item We denote by $\Phi=\Phi_0+\sum_{i,j=1,\dots, r}\Phi_{i,j}$ the decomposition of $\Phi$ corresponding to the splitting $E=M_1\oplus\cdots\oplus M_r$. 
\item We normalize the volume of $X$ as 1.
\end{itemize}
Then the following holds:
\begin{prop}\label{non-holomorphic} {\it Let $h\in\diag_M(\H_E)$. Then the following holds:
\begin{align}
\M(h, K)&= \int_X(\Psi(s)(\bar{\partial}s),\bar{\partial}s)_{h_\refe}+\inum\int_X\Tr((\delta(\xi)-\overline{\delta(\xi)})\Lambda_{\omega_X}F_{h_\refe}) \notag \\ 
&+\sum_{i,j=1,\dots, r}\int_X2|\Phi_{i,j}|_{K,\omega_X}^2e^{(v_{i,j}, \xi)}+(2\pi\gamma, \overline{\delta(\xi)})+C, \label{functional}
\end{align} 
where each symbol is defined as follows:
\begin{itemize}
\item $\xi=(f_1,\dots, f_r)$ is a pair of functions satisfying $h=(e^{f_1}K_1,\dots, e^{f_r}K_r)$.
\item $s$ is the unique endomorphism satisfying $h=h_\refe(e^{s}\cdot,\cdot)$ and $\delta(\xi)$ the projection of $s$ to the diagonal part of $\End E$ of the decomposition induced from $E=L_1\oplus\cdots\oplus L_r$.
\item $\overline{\delta(\xi)}$ is the average $\int_X\delta(\xi)$ of $\delta(\xi)$.
\item $\gamma\coloneqq (\deg_{\omega_X}(L_1),\dots, \deg_{\omega_X}(L_r))$ is concerned to be a diagonal matrix.
\item The definition of $\Psi:\R\times\R\rightarrow \R$ is the same as that of \cite[P.882]{Sim1} and the definition of $\Psi(s)(\bar{\partial}s)$ is the same as that of \cite[pp.879-882]{Sim1}.
\item $C$ is a constant which is independent of $\xi$.
\end{itemize}
}
\end{prop}
\begin{proof} We decompose Donaldson's functional (see Definition \ref{Donaldson's functional}) $\M(h,K)$ as follows: 
\begin{align*}
&\M(h,K)=\M_1(h,K)+\M_2(h,K), \\
&\M_1(h,K)\coloneqq \int_0^1dt\int_X\inum \Tr(\Lambda_{\omega_X}F_{h_t}g_t^{-1}\partial_tg_t), \\
&\M_2(h,K)\coloneqq \int_0^1dt\int_X\inum \Tr(\Lambda_{\omega_X}([\Phi\wedge\Phi^{\ast h_t}])g_t^{-1}\partial_tg_t).
\end{align*}
Then the second term $\M_2(h,K)$ coincides with $\sum_{i,j=1,\dots, r}\int_X2|\Phi_{i,j}|_{K,\omega_X}^2e^{(v_{i,j}, \xi)}$. The first term $\M_1(h,K)$ is further decomposed as follows:
\begin{align*}
\M_1(h,K)=\M_1(h,h_\refe)+\M_1(h_\refe,K). 
\end{align*}
We set $C\coloneqq \M_1(h_\refe,K)$. We decompose $\M_1(h,h_\refe)$ as
\begin{align*}
\M_1(h,h_\refe)&=\int_X(\Psi(s)\bar{\partial}s,\bar{\partial}s)_{h_\refe}+\inum\int_X\Tr(s\Lambda_{\omega_X}F_{h_\refe}) \\
&=\int_X(\Psi(s)\bar{\partial}s,\bar{\partial}s)_{h_\refe}+\inum\int_X\Tr((\delta(\xi)-\overline{\delta(\xi)})\Lambda_{\omega_X}F_{h_\refe})+(2\pi\gamma, \overline{\delta(\xi)}).
\end{align*}
Then we have (\ref{functional}).
\end{proof}
If the splitting $E=M_1\oplus\cdots \oplus M_r$ coincides with the splitting $E=L_1\oplus\cdots \oplus L_r$, then $\overline{\delta(\xi)}$ corresponds with the average $\overline{\xi}$ of $\xi$, up to the average difference between $h_{\refe}$ and $K$. Under these conditions, we can evaluate the functional from below as per \cite{Miy1}. Nevertheless, if the splittings diverge even slightly, naively doing the same will result in (\ref{functional}). 

\begin{rem} For a non-holomorphic splitting, even if the Higgs field is holomorphic and $\Phi_{i,j}\neq 0$, $\log|\Phi_{i,j}|_{K,\omega_X}^2$ is not integrable in general.
\end{rem}
\medskip
\noindent
{\bf Acknowledgements.} I would like to express my gratitude to Ryushi Goto and Hisashi Kasuya for their valuable discussions and many supports. I am very grateful to Yoshinori Hashimoto for his valuable discussions, many supports, helpful comments on this paper, and for informing me of the paper \cite{BT1}. I would also like to express my gratitude to Qiongling Li for answering my questions.

\noindent
E-mail address 1: natsuo.miyatake.e8@tohoku.ac.jp

\noindent
E-mail address 2: natsuo.m.math@gmail.com \\

\noindent
Mathematical Science Center for Co-creative Society, Tohoku University, 468-1 Aramaki Azaaoba, Aoba-ku, Sendai 980-0845, Japan.

\end{document}